\documentclass[11pt]{amsart}
\usepackage{amssymb}
\theoremstyle{plain}
\newtheorem{theorem}{Theorem}
\newtheorem{corollary}{Corollary}

\newtheorem{lemma}{Lemma}
\newtheorem{proposition}{Proposition}

\theoremstyle{definition}

\theoremstyle{remark}

\numberwithin{equation}{section}

\setbox0=\hbox{$+$}
\newdimen\plusheight
\plusheight=\ht0
\def\+{\;\lower\plusheight\hbox{$+$}\;}

\setbox0=\hbox{$-$}
\newdimen\minusheight
\minusheight=\ht0
\def\-{\;\lower\minusheight\hbox{$-$}\;}

\setbox0=\hbox{$\cdots$}
\newdimen\cdotsheight
\cdotsheight=\plusheight
\def\cds{\lower\cdotsheight\hbox{$\cdots$}}

\begin{document}
\title[  points on elliptic curves over a finite prime field
]
 {Some  properties of the distribution of the numbers of
points on elliptic curves over a finite prime field
 }
\author{Saiying He }
\address{Trinity College\\
        300 Summit Street, Hartford, CT 06106-3100}
\email{Saiying.He@trincoll.edu}
\author{J. Mc Laughlin}
\address{Mathematics Department\\
 Trinity College\\
300 Summit Street, Hartford, CT 06106-3100}
\email{james.mclaughlin@trincoll.edu}

\date{August 13th, 2004}
\begin{abstract}
Let $p \geq 5$ be a prime and for $a, b \in \mathbb{F}_{p}$,
let $E_{a,b}$ denote the elliptic curve over
$\mathbb{F}_{p}$ with equation
$y^2=x^3+a\,x + b$. As usual define the trace of Frobenius $a_{p,\,a,\,b}$ by
\begin{equation*}
 \#E_{a,b}(\mathbb{F}_{p}) = p+1 -a_{p,\,a,\,b}.
\end{equation*}
We use elementary facts about exponential sums and known results about binary quadratic forms over finite fields to evaluate the
sums $\sum_{t\in\mathbb{F}_{p}} a_{p,\, t,\, b}$, $\sum _{t \in \mathbb{F}_{p}} a_{p,\,a,\, t}$,
 $ \sum_{t=0}^{p-1}a_{p,\,t,\,b}^{2}$,
$ \sum_{t=0}^{p-1}a_{p,\,a,\,t}^{2}$ and $ \sum_{t=0}^{p-1}a_{p,\,t,\,b}^{3}$
for primes $p$ in various congruence classes.

As an example of our results, we prove the following:
Let $p \equiv 5$ $($mod 6$)$ be prime and let $b \in \mathbb{F}_{p}^{*}$. Then
\begin{equation*}
\sum_{t=0}^{p-1}a_{p,\,t,\,b}^{3}=
-p\left((p-2)\left(\frac{-2}{p}\right)
+2p\right)\left(\frac{b}{p}\right).
\end{equation*}

\end{abstract}

\maketitle

\section{Introduction}
Let $p \geq 5$ be a prime and let $\mathbb{F}_{p}$ be the finite field
of $p$  elements. For $a, b \in \mathbb{F}_{p}$,
let $E_{a,b}$ denote the elliptic curve over
$\mathbb{F}_{p}$ with equation
$y^2=x^3+a\,x + b$.
Denote by $E_{a,\,b}(\mathbb{F}_{p})$ the set of $\mathbb{F}_{p}$-- rational points
on the curve $E_{a,\,b}$
and define the trace of Frobenius, $a_{p}$, by the equation
\begin{equation*}
 \#E_{a,\,b}(\mathbb{F}_{p}) = p+1 -a_{p}.
\end{equation*}
A simple counting argument makes it clear that
\begin{equation}\label{apeq}
a_{p}=- \sum_{x \in \mathbb{F}_{p}}
\left (
\frac{x^{3}+a\,x+b}{p}
\right),
\end{equation}
where $\left ( \frac{.}{p} \right )$ denotes the Legendre symbol.
We recall some of the arithmetic properties of the distribution of $a_{p}$.
The following theorem is due to Hasse \cite{H33}:
\begin{theorem}
The integer $a_{p}$ satisfies
\begin{equation*}
-2\sqrt{p} \leq a_{p} \leq 2 \sqrt{p}.
\end{equation*}
\end{theorem}
Since we wish to look at how $a_{p}$ varies as the coefficients
$a$ and $b$ of the elliptic curve vary,
 it is convenient for our purposes to write $a_{p}$ for the elliptic curve
$E_{a,b}$ as $a_{p,\,a,\,b}$.
The following result  is well known
(an easy consequence of the remarks on page 36 of
\cite{BSS99}, for example).
\begin{proposition}
Let the function $f: \mathbb{Z} \to \mathbb{N}_{0} $ be  defined  by setting
\begin{equation}\label{feq}
f(k) = \# \{(a,b) \in  \mathbb{F}_{p}^{*} \times  \mathbb{F}_{p}^{*} :
 a_{p,\,a,\,b} = k  \}.
\end{equation}
Then for each integer $k$,
\begin{equation*}
\frac{p-1}{2} \bigg | f(k).
\end{equation*}
\end{proposition}
The following result can be found in \cite{B67} (page 57).
\begin{proposition}
Define the function $f_{1}: \mathbb{Z} \to \mathbb{N}_{0} $  by
setting
\begin{equation}\label{feqa}
f_{1}(k) = \# \{(a,b) \in  \mathbb{F}_{p} \times  \mathbb{F}_{p}\setminus \{(0,0)\} :
 a_{p,\,a,\,b}=k \}.
\end{equation}
Then for each integer $k$,
\begin{equation*}
f_{1}(k)=f_{1}(-k).
\end{equation*}
\end{proposition}
The following  result
is also known (\cite{BSS99}, page 37, for example).
\begin{proposition}
Let $v$ be a quadratic non-residue modulo $p$. Then
\begin{equation*}
a_{p,\,a,\,b}= -a_{p,\,v^{2}a,\,v^{3}b}.
\end{equation*}
\end{proposition}
To better understand the distribution of the $a_{p,\, a, \, b}$
it makes sense to study the moments.
The \emph{j -invariant} of the elliptic curve $E_{a,b}$ is defined by
\[
j= \frac{2^{8}3^{3}a^{3}}{4a^{3}+27b^{2}},
\]
provided $4a^{3}+27b^{2} \not = 0$. Michel showed in \cite{M95} that if
$\{E_{a(t),\, b(t)}: t \in \mathbb{F}_{p}\}$ is a one-parameter family of
elliptic curves with $a(t)$ and $b(t)$ polynomials in $t$ such that
\[
j(t) := \frac{2^{8}3^{3}a(t)^{3}}{4a(t)^{3}+27b(t)^{2}},
\]
is non-constant, then
\[
\sum_{t   \in \mathbb{F}_{p}}a_{p,\,a(t),\,b(t)}^{2} = p^{2}+O(p^{3/2}).
\]
In \cite{B67} Birch defined
\[
S_{R}(p)= \sum_{a,\,b=0}^{p-1}
\left [ \sum_{x=0}^{p-1} \left ( \frac{x^3-ax-b}{p} \right ) \right ]^{2R}
\]
for integral $R \geq 1$, and proved
\begin{theorem}\label{tbirch}
\footnote{
In \cite{B67}, Birch incorrectly omitted the factor of $p-1$ in his statement
of Theorem \ref{tbirch}.}
For $p \geq 5$,
\begin{align*}
S_{1}(p)&=(p-1)p^2,\\
S_{2}(p)&=(p-1)(2p^3-3p),\\
S_{3}(p)&=(p-1)(5p^4-9p^2-5p),\\
S_{4}(p)&=(p-1)(14p^5-28p^3-20p^2-7p),\\
S_{5}(p)&=(p-1)( 42p^6-90p^4-75p^3-35p^2-9p-\tau (p)),
\end{align*}
where $\tau (p)$ is Ramanujan's $\tau$-function.
\end{theorem}
Theorem \ref{tbirch} evaluates sums of the form
$ \sum_{a,\,b=0}^{p-1}
a_{p,\,a,\,b}^{2R}$ in terms of $p$ and these results were derived by
Birch as consequences of  the Selberg trace formula .

In this present paper we instead use elementary facts about exponential sums
and known results about binary quadratic forms over finite fields to evaluate the
sums $\sum_{t\in\mathbb{F}_{p}} a_{p,\, t,\, b}$, $\sum _{t \in \mathbb{F}_{p}} a_{p,\,a,\, t}$,
 $ \sum_{t=0}^{p-1}a_{p,\,t,\,b}^{2}$,
$ \sum_{t=0}^{p-1}a_{p,\,a,\,t}^{2}$ and $ \sum_{t=0}^{p-1}a_{p,\,t,\,b}^{3}$,
for primes $p$ in particular congruence classes. In particular,
we prove the following theorems.
\begin{theorem}\label{t4}
Let $p \geq 5$  be a prime, and $a,\,b \in \mathbb{F}_{p}$.  Then
\begin{flushleft}
(i) $\sum_{t\in\mathbb{F}_{p}} a_{p,\, t,\, b}= -p\left (\frac{b}{p} \right )$,\\
(ii)$\sum _{t \in \mathbb{F}_{p}} a_{p,\,a,\, t}= 0.$
\end{flushleft}
\end{theorem}
This result is elementary but we prove it for
the sake of completeness.

\begin{theorem}\label{t1}
Let $p \equiv 5$ $($mod 6$)$ be prime and let $b \in \mathbb{F}_{p}^{*}$. Then
\begin{equation}\label{t1eq}
\sum_{t=0}^{p-1}a_{p,\,t,\,b}^{2}=
p\left(p-1 - \left (\frac{-1}{p} \right ) \right ).
\end{equation}
\end{theorem}
\begin{theorem}\label{t2}
Let $p \geq 5$ be prime and let $a \in \mathbb{F}_{p}^{*}$.
Then
{\allowdisplaybreaks
\begin{equation}\label{t2eq}
\sum_{t=0}^{p-1}a_{p,\,a,\,t}^{2}
=p\left (p-1-\left (\frac{-3}{p} \right)-\left(\frac{-3a}{p} \right ) \right).
\end{equation}
}
\end{theorem}
Theorem \ref{t1} and Theorem \ref{t2} could  be deduced from
Theorem \ref{tbirch}, but we believe it is of interest to give elementary
proofs that do not use the Selberg trace formula.

\begin{theorem}\label{t3}
Let $p \equiv 5$ $($mod 6$)$ be prime and let $b \in \mathbb{F}_{p}^{*}$. Then
\begin{equation*}
\sum_{t=0}^{p-1}a_{p,\,t,\,b}^{3}=
-p\left((p-2)\left(\frac{-2}{p}\right)
+2p\right)\left(\frac{b}{p}\right).
\end{equation*}
\end{theorem}

\section{Proof of the Theorems}
We introduce some standard notation.
Define $e(j/p):=\exp(2 \pi i j/p)$, so that
\begin{equation}\label{legeq}
\sum_{t=0}^{p-1}e\left (\frac{jt}{p} \right)=
\begin{cases}
p, &p\,|j,\\
0,  &(j,p)=1.
\end{cases}
\end{equation}
Define
\begin{equation}\label{Geq}
G_{p}=
\begin{cases}
\sqrt{p}, &p\equiv 1 (\mod 4),\\
i \sqrt{p},  &p\equiv 3 (\mod 4).
\end{cases}
\end{equation}
\begin{lemma}
Let
$\left ( \frac{.}{p} \right )$ denote the Legendre symbol, modulo $p$. Then
\begin{equation}\label{LeGaueq}
\left (
\frac{z}{p}
\right )
=
\frac{1}{G_{P}}
\sum_{d=1}^{p-1}
\left (
\frac{d}{p}
\right )
e\left (
\frac{dz}{p}
\right ).
\end{equation}
\end{lemma}
\begin{proof}
See \cite{BEW99}, Theorem 1.1.5 and Theorem 1.5.2.
\end{proof}
We will occasionally use the fact that if $\mathbb{H}$ is a
subset of $\mathbb{F}_{p}$,
\begin{equation}\label{legsym}
\sum_{d\in \mathbb{F}_{p} \setminus \mathbb{H}}
\left (
\frac{d}{p}
\right )=
-\sum_{d\in  \mathbb{H}}
\left (
\frac{d}{p}
\right ).
\end{equation}
We will also occasionally make use of some implications of the Law of
Quadratic Reciprocity (see \cite{IR90}, page 53, for example).
\begin{theorem}\label{tqr}
Let $p$ and $q$ be odd primes. Then\\
{\allowdisplaybreaks
$($a$)$$ \left (
\frac{-1}{p}
\right )=(-1)^{(p-1)/2}.$\\
$($b$)$ $ \left (
\frac{2}{p}
\right )=(-1)^{(p^{2}-1)/8}$.\\
$($c$)$ $\left (
\frac{p}{q}
\right )
\left (
\frac{q}{p}
\right )=(-1)^{((p-1)/2)((q-1)/2)}$.
}
\end{theorem}
We now prove Theorems  \ref{t4}, \ref{t1}, \ref{t2} and \ref{t3},

\begin{flushleft}
\textbf{Theorem \ref{t4}.}
\emph{Let $p \geq 5$  be a prime, and $a,\,b \in \mathbb{F}_{p}$.  Then}\\
\emph{
$($i$)$ $\sum_{t\in\mathbb{F}_{p}} a_{p,\, t,\, b}= -p\left (\frac{b}{p} \right )$,\\
$($ii$)$ $\sum _{t \in \mathbb{F}_{p}} a_{p,\,a,\, t}= 0.$}
\end{flushleft}

\emph{Proof.}
(i) From \eqref{apeq} and  \eqref{LeGaueq}, it follows that
\begin{equation*}
\sum_{t\in\mathbb{F}_{p}} a_{p, t, b}=
-\sum_{x\in\mathbb{F}_{p}}
\sum_{d=1}^{p-1}
\frac{1}{G_{P}}
\left (\frac{d}{p}\right )
e\left ( \frac{d(x^3+b)}{p}\right)
\sum_{t\in\mathbb{F}_{p}}
e\left (\frac{t d x}{p}\right)
\end{equation*}
The inner sum over $t$ is zero unless  $x=0$, in which case
it equals to $p$.  The left side therefore can be simplified to give
\[
\sum_{t\in\mathbb{F}_{p}} a_{p, t, b}=-\sum_{d=1}^{p-1}
\frac{p}{G_{P}}
\left (\frac{d}{p}\right )
e\left(\frac{d b}{p}\right)=-p\left (\frac{b}{p}\right ).
\]
The last equality follows from  \eqref{LeGaueq}.

$($ii$)$:  From \eqref{apeq} and \eqref{LeGaueq}, it follows that
\begin{equation*}
\sum_{t \in \,\mathbb{F}_{p}}a_{p,a,t}=
-\sum_{x\in\mathbb{F}_{p}}
\sum_{d=1}^{p-1}
\frac{1}{G_{p}}
\left (\frac{d}{p} \right )
e\left(\frac{d(x^3+a x)}{p}\right)
\sum_{t\in\mathbb{F}_{p}}
e\left (\frac{d t}{p} \right )
=0.
\end{equation*}
The inner sum over $t$ is equal to $0$, by \eqref{legeq}, since $1 \leq d \leq p-1$.
\begin{flushright}
$\Box$
\end{flushright}
The result at (ii) follows also, in the case of primes
$p \equiv 3$ (mod 4), from the fact that $a_{p,\,a,\, t}= -a_{p,\,a,\, p-t}$.
However, this is not the case for  primes
$p \equiv 1$ (mod 4). For example,
 \[
\{a_{13,\,1,\,t}: 0 \leq t \leq 12\} =\{-6, -4, 2, -1, 0, 5, 1, 1, 5, 0, -1, 2, -4\}.
\]
The results in Theorem \ref{t4} are almost certainly known, although we have
not been able to find a reference.

\textbf{Theorem \ref{t1}.}
\emph{Let $p \equiv 5$ $($mod 6$)$ be prime and let $b \in \mathbb{F}_{p}^{*}$. Then}
\begin{equation*}
\sum_{t=0}^{p-1}a_{p,\,t,\,b}^{2}=
p\left(p-1 - \left (\frac{-1}{p} \right ) \right ).
\end{equation*}

\emph{Proof.}
 From \eqref{apeq} and  \eqref{LeGaueq} it follows that
\begin{multline*}
\sum_{t \in \mathbb{F}_{p}}a_{p,t,b}^{2}
=
\frac{1}{G_{p}^{2}}\sum_{d_{1},d_{2}=1}^{p-1}\left (
\frac{d_{1}d_{2}}{p}
\right)
\sum_{x_{1},x_{2} \in \mathbb{F}_{p}}
e\left (
\frac{d_{1}(x_{1}^{3}+b)+d_{2}(x_{2}^{3}+b)}{p}
\right)\\
\times
\sum_{t \in \mathbb{F}_{p}}
e\left (
\frac{t(d_{1}x_{1}+d_{2}x_{2})}{p}
\right).
\end{multline*}
The inner sum over $t$ is zero, unless
$x_{1} \equiv -d_{1}^{-1}d_{2}x_{2}(\text{mod }p)$, in which case it equals $p$.
Thus
\begin{equation*}
\sum_{t \in \mathbb{F}_{p}}a_{p,t,b}^{2}=
\frac{p}{G_{p}^{2}}\sum_{d_{1},d_{2}=1}^{p-1}
\left (
\frac{d_{1}d_{2}}{p}
\right)e\left (
\frac{b(d_{1}+d_{2})}{p}
\right)
\sum_{x_{2} \in \mathbb{F}_{p}}
e\left (
\frac{d_{1}^{-2}d_{2}x_{2}^{3}(d_{1}^{2}-d_{2}^{2})}{p}
\right).
\end{equation*}
Since the map $x \to x^{3}$ is one-to-one on $F_{p}$, when
$p \equiv 5 (\text{mod }6)$, the $x_{2}^{3}$ in the inner sum can be replaced
by $x_{2}$. Thus the inner sum is zero unless
$d_{2}^{2}-d_{1}^{2}\equiv 0 (\text{mod }p)$, in which case it equals $p$.
It follows that
\begin{align*}
\sum_{t \in \mathbb{F}_{p}}a_{p,t,b}^{2}
&=
\frac{ p^{2} }{ G_{p}^{2} }
\left (
\sum_{d_{1}=1}^{p-1}
\left (
\frac{ d_{1}^{2} }{p}
\right)
e\left (\frac{
b(2d_{1})}{p}
\right)
+
\sum_{ d_{1}=1}^{p-1}
\left (
\frac{ -d_{1}^{2} }{p}
\right)
e\left (
\frac{b( d_{1}-d_{1} )}{p}
\right )
\right )\\
&=
\frac{ p^{2} }{ G_{p}^{2} }
\left (
-1+
\left (
\frac{ -1 }{p}
\right)(p-1)
\right )=
\frac{ p^{2} }{ G_{p}^{2} }
\left (
\frac{ -1 }{p}
\right)
\left (p-1
-
\left (
\frac{ -1 }{p}
\right)
\right ).
\end{align*}
We have once again used \eqref{LeGaueq} to compute the sums, noting that
the sums above start with $d_{1}=1$.
The result now follows since $p/G_{p}^{2} \times (-1|p)=1$ for all primes $p \geq 3$.
\begin{flushright}
$\Box$
\end{flushright}

Remarks: (1) It is clear that the results will remain true if
$a(t)=t$ is replaced by any function $a(t)$ which is one-to-one on $F_{p}$.\\
(2) It is more difficult to
determine the values taken by $\sum_{t \in \mathbb{F}_{p}}a_{p,t,b}^{2}$ for
primes $p\equiv 1 (\text{mod } 6)$. This is principally because the
map $x \to x^{3}$ is not one-to-one on $F_{p}$ for these primes
(so that \eqref{legeq} cannot be used so easily to simplify the summation),
but also because
the answer depends on which coset $b$ belongs to in
$\mathbb{F}_{p}^{*}/\mathbb{F}_{p}^{*3}$.

Before proving the next theorem, it is necessary to recall a result  about
quadratic forms over finite fields. Let $q$ be a power of an odd prime and
let $\eta$ denote the quadratic character on $\mathbb{F}_{q}^{*}$ (so that if $q=p$, an
odd prime, then $\eta (c) = \left ( c /p\right)$, the Legendre symbol).
The function $v$ is defined on $\mathbb{F}_{q}$ by
\begin{equation}\label{veq}
v(b)=
\begin{cases}
-1,& b \in \mathbb{F}_{q}^{*},\\
q-1,& b=0.
\end{cases}
\end{equation}
Suppose
\[
f(x_{1},\dots,x_{n}) = \sum_{i,j=1}^{n}a_{ij}x_{i}x_{j},
\hspace{15pt}
\text{ with } a_{ij}=a_{ji},
\]
is a quadratic form over $\mathbb{F}_{q}$, with associated matrix
$A=(a_{ij})$ and let $\triangle$ denote the determinant of $A$
($f$ is \emph{non-degenerate} if  $\triangle \not = 0$).
\begin{theorem}
Let $f$ be a non-degenerate quadratic form over  $\mathbb{F}_{q}$,
$q$ odd, in an even number $n$ of indeterminates. Then for $b \in
\mathbb{F}_{q}$ the number of solutions of the equation
$f(x_{1},\dots,x_{n}) =b$ in $\mathbb{F}_{q}^{\,n}$ is
\begin{equation}\label{qfeq}
q^{n-1}+v(b)q^{(n-2)/2}\eta \left ( (-1)^{n/2} \triangle \right ).
\end{equation}
\end{theorem}
\begin{proof}
See \cite{LN97}, pp 282--293.
\end{proof}

\textbf{Theorem \ref{t2}.}
\emph{Let $p \geq 5$ be prime and let $a \in \mathbb{F}_{p}^{*}$.
Then}
{\allowdisplaybreaks
\begin{equation*}
\sum_{t=0}^{p-1}a_{p,\,a,\,t}^{2}
=p\left (p-1-\left (\frac{-3}{p} \right)-\left(\frac{-3a}{p} \right ) \right).
\end{equation*}
}

\emph{Proof.}
Once again
\eqref{apeq} and  \eqref{LeGaueq} give that
\begin{multline*}
\sum_{t \in \mathbb{F}_{p}}a_{p,\,a,\,t}^{2}
=
\frac{1}{G_{p}^{2}}\sum_{d_{1},d_{2}=1}^{p-1}\left (
\frac{d_{1}d_{2}}{p}
\right)
\sum_{x_{1},x_{2} \in \mathbb{F}_{p}}
e\left (
\frac{d_{1}(x_{1}^{3}+ax_{1})+d_{2}(x_{2}^{3}+a x_{2})}{p}
\right)\\
\times
\sum_{t \in \mathbb{F}_{p}}
e\left (
\frac{t\,b(d_{1}+d_{2})}{p}
\right).
\end{multline*}
The inner sum over $t$ is zero, unless
$d_{1} \equiv -d_{2}(\text{mod }p)$, in which case it equals $p$.
Thus
\begin{align}\label{Gpeq}
\sum_{t \in \mathbb{F}_{p}}a_{p,\,a,\,t}^{2}
&=
\frac{p}{G_{p}^{2}}
\left (
\frac{-1}{p}
\right)
\sum_{x_{1},\,x_{2} \in \mathbb{F}_{p}}
\sum_{d_{1}=1}^{p-1}
e\left (
\frac{d_{1}(x_{1}^{3}+a\,x_{1}-x_{2}^{3}-a\,x_{2})}{p}
\right)\\
&=
\sum_{x_{1},\,x_{2} \in \mathbb{F}_{p}}
\sum_{d_{1}=1}^{p-1}
e\left (
\frac{d_{1}(x_{1}-x_{2})(x_{1}^{2}+x_{1}x_{2}+x_{2}^2+a)}{p}
\right). \notag
\end{align}
We have used the fact that $p/G_{p}^{2} \times (-1|p)=1$ for all
primes $p \geq 3$. The inner sum over $d_{1}$ equals $-1$, unless
one of the factors $x_{1}-x_{2}$, $x_{1}^{2}+x_{1}x_{2}+x_{2}^2+a$
equals $0$, in which case the sum is $p-1$. The equation
$x_{1}=x_{2}$ has $p$ solutions and, by \eqref{qfeq} with $q=p$,
$n=2$, $f(x_{1},x_{2})= x_{1}^{2}+x_{1}x_{2}+x_{2}^{2}$ and $A =
\left (
\begin{smallmatrix}
1& (p+1)/2\\
(p+1)/2&1
\end{smallmatrix}
\right )$, the equation $x_{1}^{2}+x_{1}x_{2}+x_{2}^{2}=-a$ has
\begin{equation*}
p+(-1)
\left (
\frac{-1(1-(p+1)^2/4)}{p}
\right)
= p-
\left (
\frac{-3}{p}
\right)
\end{equation*}
solutions. However, we need to be careful to avoid double counting
and to examine when $x_{1}^{2}+x_{1}x_{2}+x_{2}^{2}=-a$ has a
solution with $x_{1}=x_{2}$. The equation $3x_{1}^{2}=-a$ will
have two solutions if $\left ( \frac{-3a}{p} \right)=1$ and none
if $\left ( \frac{-3a}{p} \right)=-1$. Hence the number of
solutions to the equation $3x_{1}^{2}=-a$ is $\left (
\frac{-3a}{p} \right)+1$. Thus the number of solutions to
$(x_{1}-x_{2})(x_{1}^{2}+x_{1}x_{2}+x_{2}^2+a)=0$ is
{\allowdisplaybreaks
\[
p+ \left( p-
\left (
\frac{-3}{p}
\right)
\right) -
\left(
\left (
\frac{-3a}{p}
\right)+1
\right)
=2p-1-\left (
\frac{-3}{p}
\right)
-\left (
\frac{-3a}{p}
\right).
\]
}
Thus
\begin{multline*}
\sum_{t \in \mathbb{F}_{p}}a_{p,\,a,\,t}^{2}
=\left (2p-1-\left (
\frac{-3}{p}
\right)
-\left (
\frac{-3a}{p}
\right)
\right)(p-1) \\
+\left ( p^2-\left (2p-1-\left (
\frac{-3}{p}
\right)
-\left (
\frac{-3a}{p}
\right)
\right)
\right)(-1).
\end{multline*}
The right side now simplifies to give the result.
\begin{flushright}
$\Box$
\end{flushright}

Before proving Theorem \ref{t3}, we need some preliminary lemmas.

\begin{lemma}\label{l1}
Let  $p \equiv 5$ $($mod 6$)$ be prime. Then
{\allowdisplaybreaks
\begin{multline}\label{l1eq}
\sum_{d, e, f =1}^{p-1}
\left (
\frac{ef(1+e+f)}{p}
\right)\sum_{ y, z \in \mathbb{F}_{p}}
e\left (
\frac{d(-(e y+fz)^{3}+e y^{3}+fz^{3})}{p}
\right)\\
=-p(p-1)
\left (
1+\left (
\frac{-1}{p}
\right)
\right)\\
+\sum_{d, e, f =1}^{p-1}
\left (
\frac{e+ef+f}{p}
\right)
\sum_{ y,z \in \mathbb{F}_{p}}
e\left (
\frac{d\, f z(-f^2(y+1)^{3}+e^2 y^{3}+1)}{p}
\right).
\end{multline}
}
\end{lemma}

\begin{proof}
If $z=0$, the left side of \eqref{l1eq} becomes
{\allowdisplaybreaks
\begin{align*}
S_{0}:&=\sum_{d, e, f =1}^{p-1}
\left (
\frac{ef(1+e+f)}{p}
\right)\sum_{y \in \mathbb{F}_{p}}
e\left (
\frac{dy^3e(1-e^2)}{p}
\right)\\
&=(p-1)\sum_{e, f =1}^{p-1}
\left (
\frac{ef(1+e+f)}{p}
\right)\sum_{y \in \mathbb{F}_{p}}
e\left (
\frac{ye(1-e^2)}{p}
\right)\\
&=p(p-1)
\left (\sum_{f =1}^{p-1}
\left (
\frac{f(2+f)}{p}
\right)
+\sum_{f =1}^{p-1}
\left (
\frac{-f^2}{p}
\right)
\right)\\
&=p(p-1)
\left (\sum_{f =1}^{p-1}
\left (
\frac{2f^{-1}+1}{p}
\right)
+
\sum_{f =1}^{p-1}
\left (
\frac{-1}{p}
\right)
\right)\\
&=p(p-1)
\left (
-1
+
(p-1)\left (
\frac{-1}{p}
\right)
\right).
\end{align*}
}
The second equality follows since
$\{y^3:y \in \mathbb{F}_{p}\}=\{y:y \in \mathbb{F}_{p}\}$ for the primes
$p$ being considered, the third equality follows from \eqref{legeq} and the last
equality follows  from \eqref{legsym}.

If $z \not = 0$, then the left side of \eqref{l1eq} equals
{\allowdisplaybreaks
\begin{multline}\label{l1eq2}
S_{1}:=\sum_{d, e, f ,z=1}^{p-1}
\left (
\frac{ef(1+e+f)}{p}
\right)\sum_{ y \in \mathbb{F}_{p}}
e\left (
\frac{d(-(ey+fz)^{3}+e y^{3}+fz^{3})}{p}
\right)=\\
\sum_{d, e, f ,z=1}^{p-1}
\left (
\frac{ef(1+e+f)}{p}
\right)
\sum_{ y \in \mathbb{F}_{p}}
e\left (
\frac{dz^3(-(eyz^{-1}+f)^{3}+e (yz^{-1})^{3}+f)}{p}
\right).
\end{multline}
}
Now replace $y$ by $yz$ and then $z^{3}$ by $z$ (justified by the same argument as above)
 and finally $e$ by $ef$ to get this last sum equals
{\allowdisplaybreaks
\begin{equation*}
\sum_{d, e, f ,z=1}^{p-1}
\left (
\frac{e(1+ef+f)}{p}
\right)
\times
\sum_{ y \in \mathbb{F}_{p}}
e\left (
\frac{dfz(-f^2(ey+1)^{3}+e y^{3}+1)}{p}
\right).
\end{equation*}
}
We wish to extend the last sum to include $z=0$. If we set $z=0$ on
the right side  of the last equation (and denote the resulting sum by "r.s.")
and sum over $d$ and $y$ we get that
{\allowdisplaybreaks
\begin{align*}
r.s.&=p(p-1)
\sum_{ e, f=1 }^{p-1}
\left (
\frac{e(1+f(e+1))}{p}
\right)\\
&=p(p-1)
\left (
\sum_{  f=1 }^{p-1}
\left (
\frac{-1}{p}
\right)
+
\sum_{  e=1}^{p-2}
\sum_{  f =1}^{p-1}
\left (
\frac{e(1+f(e+1))}{p}
\right)
\right ).
\end{align*}
}
Replace $f$ by $f(e+1)^{-1}$ in the second sum above and then
{\allowdisplaybreaks
\begin{align*}
r.s.&=p(p-1)
\left (
(p-1)
\left (
\frac{-1}{p}
\right)
+
\sum_{  e=1}^{p-2}
\sum_{  f =1}^{p-1}
\left (
\frac{e(1+f)}{p}
\right)
\right )\\
&=p(p-1)
\left (
(p-1)
\left (
\frac{-1}{p}
\right)
+
\sum_{  e=1}^{p-2}
\left (
\frac{e}{p}
\right)
\sum_{  f =1}^{p-1}
\left (
\frac{1+f}{p}
\right)
\right )\\
&=p(p-1)
\left (
(p-1)
\left (
\frac{-1}{p}
\right)
+
\left (-
\left (
\frac{-1}{p}
\right)
\right )
(-1)\right )\\
&=p^{2}(p-1)
\left (
\frac{-1}{p}
\right)
.
\end{align*}
}
It follows that the left side of \eqref{l1eq2} equals
{\allowdisplaybreaks
\begin{multline*}
-p^{2}(p-1)
\left (
\frac{-1}{p}
\right)\\
+ \sum_{d, e, f =1}^{p-1} \left ( \frac{e(1+ef+f)}{p} \right)
\sum_{ y,z \in \mathbb{F}_{p}} e\left ( \frac{d f z (-f^2(e
y+1)^{3}+e y^{3}+1)}{p} \right),
\end{multline*}
}
and thus that the left side of \eqref{l1eq} equals
{\allowdisplaybreaks
\begin{align}\label{seq2}
S_{0}+S_{1}
&=-p(p-1)
\left (
1+\left (
\frac{-1}{p}
\right)
\right)
+\sum_{d, e, f =1}^{p-1}
\left (
\frac{e(1+ef+f)}{p}
\right)\\
&\phantom{sasadadasdsasadada}
\times
\sum_{ y,z \in \mathbb{F}_{p}}
e\left (
\frac{dfz(-f^2(ey+1)^{3}+e y^{3}+1)}{p}
\right)\notag  \\
&=-p(p-1)
\left (
1+\left (
\frac{-1}{p}
\right)
\right)
+\sum_{d, e, f =1}^{p-1}
\left (
\frac{e+ef+f}{p}
\right)\notag \\
&\phantom{sasadadasdsasadada}
\times
\sum_{ y,z \in \mathbb{F}_{p}}
e\left (
\frac{dfz(-f^2(y+1)^{3}+e^2 y^{3}+1)}{p}
\right).
\notag
\end{align}
}
The second equality in \eqref{seq2}
follows upon replacing $y$ by $ye^{-1}$ and then
$e$ by $e^{-1}$.
\end{proof}

\begin{lemma}\label{l1in}
Let  $p \equiv 5$ $($mod 6$)$ be prime. Then
{\allowdisplaybreaks
\begin{align}\label{s*eq}
S^{*}&:=\sum_{d, e, f =1}^{p-1}
\left (
\frac{e+ef+f}{p}
\right)
\sum_{ y,z \in \mathbb{F}_{p}}
e\left (
\frac{dfz(-f^2(y+1)^{3}+e^2 y^{3}+1)}{p}
\right)\\
&=2p(p-1)
\left (
 -1+(p-1)\left (
\frac{-1}{p}
\right)
\right)
-3p(p-1)
\left (
\frac{-2}{p}
\right)
+
p(p-1)S^{**}, \notag
\end{align}
}
 where
{\allowdisplaybreaks
\begin{equation*}
S^{**}:=\sum_{e,\,f=2,\, \,e^{2} \not = f^{2}}^{p-2}
\left (
\frac{(1+e)(e-f)(1+e-f)(-1+f)}{p}
\right).
\end{equation*}
}

\end{lemma}

\begin{proof}
Upon changing the order of summation slightly, we get that
\begin{align*}
S^{*}&=\sum_{e, f =1}^{p-1}
\left (
\frac{e+ef+f}{p}
\right)\sum_{d =1}^{p-1}
\sum_{ y,z \in \mathbb{F}_{p}}
e\left (
\frac{dfz(-f^2(y+1)^{3}+e^2 y^{3}+1)}{p}
\right)
\end{align*}
If $y=0$, the inner double sum over $d$ and $z$ is zero, unless $f=\pm1$,
if which case it equals $p(p-1)$ and the right side of \eqref{s*eq} equals
{\allowdisplaybreaks
\[
p(p-1)
\left (
\sum_{ e=1}^{p-1}
\left (
\frac{2e+1}{p}
\right)
+
\sum_{ e=1}^{p-1}
\left (
\frac{-1}{p}
\right)
\right)
=p(p-1)
\left (
 -1+(p-1)\left (
\frac{-1}{p}
\right)
\right).
\]
}
By similar reasoning, if $y=-1$, the right side of \eqref{s*eq} also equals
{\allowdisplaybreaks
\[
p(p-1)
\left (
 -1+(p-1)\left (
\frac{-1}{p}
\right)
\right).
\]
}
Thus
{\allowdisplaybreaks
\begin{align}\label{s*eq2}
S^{*}&=2p(p-1)
\left (
 -1+(p-1)\left (
\frac{-1}{p}
\right)
\right)\\
&+
\sum_{y =1}^{p-2}
\sum_{ e, f =1}^{p-1}
\left (
\frac{e+ef+f}{p}
\right)
\sum_{d =1}^{p-1}
\sum_{ z \in \mathbb{F}_{p}}
e\left (
\frac{dfz(-f^2(y+1)^{3}+e^2 y^{3}+1)}{p}
\right) \notag
\\&=2p(p-1)
\left (
 -1+(p-1)\left (
\frac{-1}{p}
\right)
\right)
+
\sum_{y =1}^{p-2}
\left (
\frac{y(y+1)}{p}
\right)\notag
\\
&\times
\sum_{e, f =1}^{p-1}
\left (
\frac{(e+f)y+e(1+f)}{p}
\right)
\sum_{d =1}^{p-1}
\sum_{ z \in \mathbb{F}_{p}}
e\left (
\frac{dfz((e^{2}-f^2)y+1-f^{2})}{p}
\right),\notag
\end{align}
}
where the last equality follows upon replacing $f$ by $f(y+1)^{-1}$
and $e$ by $ey^{-1}$. The inner sum over $d$ and $z$ is zero unless
\[
(e^{2}-f^2)y+1-f^{2}=0,
\]
in which case the inner sum is $p(p-1)$. We distinguish the cases
$e^{2}=f^{2}$ and $e^{2} \not =f^{2}$. If $e^{2}=f^{2}$, then necessarily
$e^{2}=f^{2}=1$ and the sum on the right side of \eqref{s*eq2} becomes
\begin{multline}\label{e2eqf2}
p(p-1)\sum_{y =1}^{p-2}
\left (
\frac{y(y+1)}{p}
\right)
\left(
\left (
\frac{2(y+1)}{p}
\right)
+
\left (
\frac{0}{p}
\right)
+
\left (
\frac{-2}{p}
\right)
+
\left (
\frac{-2y}{p}
\right)
\right)\\
=-3p(p-1)
\left (
\frac{-2}{p}
\right).
\end{multline}
If $e^{2} \not =f^{2}$ then
\[
y=\frac{f^{2}-1}{e^{2}-f^{2}},
\]
and since $y \not =0, -1$, we exclude $f^{2}=1$ and $e^{2}=1$. After substituting
for $y$ in the sum in the final expression in \eqref{s*eq2}, we find that
\begin{equation}\label{se*eql}
S^{*}=2p(p-1)
\left (
 -1+(p-1)\left (
\frac{-1}{p}
\right)
\right)
-3p(p-1)
\left (
\frac{-2}{p}
\right)
+
p(p-1)S^{**},
\end{equation} where
{\allowdisplaybreaks
\begin{equation}\label{s**eq}
S^{**}:=\sum_{e,\,f=2,\, \,e^{2} \not = f^{2}}^{p-2}
\left (
\frac{(1+e)(e-f)(1+e-f)(-1+f)}{p}
\right).
\end{equation}
}

\end{proof}

\begin{lemma}\label{l2}
Let  $p \equiv 5$ $($mod 6$)$ be prime and let $S^{**}$ be as defined in Lemma \ref{l1in}. Then
{\allowdisplaybreaks
\begin{multline*}
S^{**}=\sum_{e=0}^{p-1}
\sum_{f=0}^{p-1}
\left (
\frac{(1+e)(e-f)(1+e-f)(-1+f)}{p}
\right)\\
+2\left (
\frac{-6}{p}
\right)
+3\left (
\frac{-2}{p}
\right)
+
3\left (
\frac{-1}{p}
\right)
+2.
\end{multline*}
}
\end{lemma}

\begin{proof}
Clearly we
can remove the restrictions $f\not =e$, $f \not =1$ and $e \not = -1$
freely. If we set $f=-e$, we have that
\begin{align*}
\sum_{e,\,f=2,\, \,e=- f}^{p-2}&
\left (
\frac{(1+e)(e-f)(1+e-f)(-1+f)}{p}
\right)
=\sum_{e=2}^{p-2}
\left (
\frac{-2e(1+2e)}{p}
\right)\\
&=-\left (
\left (
\frac{-6}{p}
\right)
+\left (
\frac{-2}{p}
\right)
+
\left (
\frac{-1}{p}
\right)
\right ).
\end{align*}
The last equality follows from \eqref{legsym}. Thus
{\allowdisplaybreaks
\begin{multline*}
S^{**}=\sum_{e,\,f=2}^{p-2}
\left (
\frac{(1+e)(e-f)(1+e-f)(-1+f)}{p}
\right)\\
+\left (
\frac{-6}{p}
\right)
+\left (
\frac{-2}{p}
\right)
+
\left (
\frac{-1}{p}
\right).
\end{multline*}
}
If $f$ is set equal to 0 in the sum above we get
\[
\sum_{e=2}^{p-2}
\left (
\frac{-e}{p}
\right)
=-1 -\left (
\frac{-1}{p}
\right).
\]
If $f$ is set equal to -1 in this sum  we get
\begin{multline*}
\sum_{e=2}^{p-2}
\left (
\frac{-2(2+e)}{p}
\right)
=-\left (
\left (
\frac{-4}{p}
\right)
+
\left (
\frac{-2}{p}
\right)
+
\left (
\frac{-6}{p}
\right)
\right)\\
=-\left (
\left (
\frac{-1}{p}
\right)
+
\left (
\frac{-2}{p}
\right)
+
\left (
\frac{-6}{p}
\right)
\right).
\end{multline*}
Thus
{\allowdisplaybreaks
\begin{multline*}
S^{**}=\sum_{e=2}^{p-2}
\sum_{f=0}^{p-1}
\left (
\frac{(1+e)(e-f)(1+e-f)(-1+f)}{p}
\right)\\
+2\left (
\left (
\frac{-6}{p}
\right)
+\left (
\frac{-2}{p}
\right)
+
\left (
\frac{-1}{p}
\right)
\right)
+1 +\left (
\frac{-1}{p}
\right).
\end{multline*}
}
If we set $e=0$ in this latest sum we get
\[
\sum_{f=0}^{p-1}
\left (
\frac{-f(1-f)(-1+f)}{p}
\right)
=
\sum_{f=0, f \not = 1}^{p-1}
\left (
\frac{f}{p}
\right)=-1.
\]
If we set $e=1$ in this  sum we get
\[
\sum_{f=0}^{p-1}
\left (
\frac{2(1-f)(2-f)(-1+f)}{p}
\right)
=
\sum_{f=0, f \not = 1}^{p-1}
\left (
\frac{-2(2-f)}{p}
\right)=-\left (
\frac{-2}{p}
\right).
\]
Thus
{\allowdisplaybreaks
\begin{multline*}
S^{**}=\sum_{e=0}^{p-1}
\sum_{f=0}^{p-1}
\left (
\frac{(1+e)(e-f)(1+e-f)(-1+f)}{p}
\right)\\
+2\left (
\frac{-6}{p}
\right)
+3\left (
\frac{-2}{p}
\right)
+
3\left (
\frac{-1}{p}
\right)
+2.
\end{multline*}
}
\end{proof}

\begin{lemma}\label{l3}
Let $p \equiv 5 ($ mod $6)$ be prime. Then
\begin{equation*}
\sum_{e=0}^{p-1}
\sum_{f=0}^{p-1}
\left (
\frac{(1+e)(e-f)(1+e-f)(-1+f)}{p}
\right)\\=p\left (
\frac{2}{p}
\right)
+1.
\end{equation*}
\end{lemma}

\begin{proof}
If $f$ is replaced by $f+1$ and then $e$ is replaced by $e+f$, the
value of the double sum above does not change. Thus
{\allowdisplaybreaks
\begin{align}\label{nfs}
\sum_{e=0}^{p-1}
\sum_{f=0}^{p-1}&
\left (
\frac{(1+e)(e-f)(1+e-f)(-1+f)}{p}
\right)\\
& =
\sum_{e=0}^{p-1}
\sum_{f=0}^{p-1}
\left (
\frac{(1+e)(e-f-1)(e-f)f}{p}
\right) \notag \\
& =
\sum_{e=0}^{p-1}
\sum_{f=0}^{p-1}
\left (
\frac{(1+e+f)(e-1)e f}{p}
\right) \notag \\
& =
\sum_{e=0}^{p-1}
\sum_{f=0}^{p-1}
\left (
\frac{e(e-1)}{p}
\right)
\sum_{f=0}^{p-1}
\left (
\frac{(1+e+f) f}{p}
\right).\notag
\end{align}
}
We evaluate the inner sum using \eqref{LeGaueq}.
{\allowdisplaybreaks
\begin{align*}
\sum_{f=0}^{p-1}&
\left (
\frac{(1+e+f) f}{p}
\right)
=\frac{1}{G_{p}^{2}}
\sum_{f=0}^{p-1}
\sum_{d_{1},d_{2}=1}^{p-1}
\left (
\frac{d_{1}d_{2}}{p}
\right)
e\left (
\frac{d_{1}f+d_{2}(1+e+f)}{p}
\right)\\
&=\frac{1}{G_{p}^{2}}
\sum_{d_{1},d_{2}=1}^{p-1}
\left (
\frac{d_{1}d_{2}}{p}
\right)
e\left (
\frac{d_{2}(1+e)}{p}
\right)
\sum_{f=0}^{p-1}
e\left (
\frac{f(d_{1}+d_{2})}{p}
\right)
\\
&=\frac{p}{G_{p}^{2}}
\sum_{d_{2}=1}^{p-1}
\left (
\frac{-1}{p}
\right)
e\left (
\frac{d_{2}(1+e)}{p}
\right)\\
&=\frac{p}{G_{p}^{2}}
\left (
\frac{-1}{p}
\right)
\sum_{d_{2}=1}^{p-1}
e\left (
\frac{d_{2}(1+e)}{p}
\right).
\end{align*}
}
The next-to-last equality follows since the sum over $f$ in the previous expression
is 0, unless $d_{1}=-d_{2}$, in which case this sum is $p$. The sum over
$d_{2}$ equals $p-1$ if $e=p-1$ and equals $-1$ otherwise. Hence the sum at \eqref{nfs}
equals
{\allowdisplaybreaks
\begin{align*}
\frac{p}{G_{p}^{2}}
\left (
\frac{-1}{p}
\right)
&\left ( \sum_{e=0}^{p-2}
\left (
\frac{e(e-1)}{p}
\right)
(-1) + (p-1)\left (
\frac{2}{p}
\right)
\right)\\
&=\frac{p}{G_{p}^{2}}
\left (
\frac{-1}{p}
\right)
\left (
\left (
\frac{2}{p}
\right)
+1 + (p-1)\left (
\frac{2}{p}
\right)
\right)\\
&=\frac{p}{G_{p}^{2}}
\left (
\frac{-1}{p}
\right)
\left (
p\left (
\frac{2}{p}
\right)
+1
\right)=p\left (
\frac{2}{p}
\right)
+1,
\end{align*}
}
the last equality following from the remark after \eqref{Gpeq}.
\end{proof}

\begin{corollary}\label{c1}
Let $S^{*}$ and $S^{**}$ be as defined in Lemma \ref{l1in}. Then
\begin{align*}
&(i)&
&S^{**}
=(p-2)\left (
\frac{2}{p}
\right)
+3\left (
\frac{-2}{p}
\right)
+
3\left (
\frac{-1}{p}
\right)
+3,\\
&(ii)& &S^{*}=p(p-1)\left (1+(2p+1)
\left (
\frac{-1}{p}
\right)
+(p-2)
\left (
\frac{2}{p}
\right)
\right).
\end{align*}
\end{corollary}
\begin{proof}
Lemmas \ref{l2} and \ref{l3} and the fact that $(-3|p)=-1$ if $p \equiv 5 (\,mod \,6)$ give
(i). Lemma \ref{l1in} and part (i) give (ii).
\end{proof}

\textbf{Theorem \ref{t3}.}
\emph{Let $p \equiv 5$ $($mod 6$)$ be prime and let $b \in \mathbb{F}_{p}^{*}$. Then}
\begin{equation}\label{t3eq}
\sum_{t=0}^{p-1}a_{p,\,t,\,b}^{3}=
-p\left((p-2)\left(\frac{-2}{p}\right)
+2p\right)\left(\frac{b}{p}\right).
\end{equation}

\emph{Proof.}
Let $g$ be a generator of $\mathbb{F}_{p}^{*}$.
It is a simple matter to show, using \eqref{apeq}, that
\[
\sum_{t=0}^{p-1}a_{p,\,t,\,b}^{3} = -\sum_{t=0}^{p-1}a_{p,\,t,\,bg}^{3}.
\]
Thus the statement at \eqref{t3eq}
is equivalent to the statement
{\allowdisplaybreaks
\begin{equation}\label{t3eqa}
\sum_{b=1}^{p-1}\sum_{t=0}^{p-1}a_{p,\,t,\,b}^{3}\left(\frac{b}{p}\right)=
-p(p-1)\left((p-2)\left(\frac{-2}{p}\right)
+2p\right).
\end{equation}
}
Let $S$ denote the left side of \eqref{t3eqa}. From \eqref{apeq} and
\eqref{LeGaueq} it follows that
{\allowdisplaybreaks
\begin{align*}
S&=- \sum_{b=1}^{p-1}\sum_{t=0}^{p-1}
\sum_{x, y, z  \in \mathbb{F}_{p}}
\left (
\frac{x^{3}+t\,x+b}{p}
\right)
\left (
\frac{y^{3}+t\,y+b}{p}
\right)
\left (
\frac{z^{3}+t\,z+b}{p}
\right)
\left(\frac{b}{p}\right)\\
&=
-\frac{1}{G_{p}^{3}}\sum_{d, e, f =1}^{p-1}\left (
\frac{d e f}{p}
\right)
\sum_{x, y, z, t \in \mathbb{F}_{p}}
e\left (
\frac{d(x^{3}+t x)+e(y^{3}+ t y)+f(z^{3}+t z)}{p}
\right)\\
&\phantom{asdsdsdsddsadsadasdsdsdsddsadsad}\times
\sum_{b \in \mathbb{F}_{p}^{*}}
\left(\frac{b}{p}\right)
e\left (
\frac{b(d+e+f)}{p}
\right)\\
&=
-\frac{1}{G_{p}^{2}}\sum_{d, e, f =1}^{p-1}\left (
\frac{d e f}{p}
\right)
\left (
\frac{d+e+f}{p}
\right)\\
&\phantom{asdsdsdsddsads}
\times \sum_{x, y, z, t \in \mathbb{F}_{p}}
e\left (
\frac{d(x^{3}+t x)+e(y^{3}+ t y)+f(z^{3}+t z)}{p}
\right)\\
&=
-\frac{1}{G_{p}^{2}}\sum_{d, e, f =1}^{p-1}\left (
\frac{d e f}{p}
\right)
\left (
\frac{d+e+f}{p}
\right)\sum_{x, y, z \in \mathbb{F}_{p}}
e\left (
\frac{dx^{3}+ey^{3}+fz^{3}}{p}
\right)
\\
&\phantom{asdsdsdsddsadsasdsdsdsddsadsasdsdsdsdds}
\times
\sum_{ t \in \mathbb{F}_{p}}
e\left (
\frac{t(d x+e y+f z)}{p}
\right)
\end{align*}
}
The inner sum is zero, unless $d x+e y+f z=0$ in $\mathbb{F}_{p}$,
in which case it equals $p$. Upon letting $x=-d^{-1}(e y+f z)$, replacing
$e$ by $d e$ and $f$ by $f e$, we get that
\begin{align*}
S&=-\frac{p}{G_{p}^{2}}\sum_{d, e, f =1}^{p-1}
\left (
\frac{ef(1+e+f)}{p}
\right)\sum_{ y, z \in \mathbb{F}_{p}}
e\left (
\frac{d(-(e y+fz)^{3}+e y^{3}+fz^{3})}{p}
\right)
\\
&=\frac{p^2(p-1)}{G_{p}^{2}}
\left (
1+\left (
\frac{-1}{p}
\right)
\right) \notag\\
&\phantom{ads}-\frac{p}{G_{p}^{2}}
\sum_{d, e, f =1}^{p-1}
\left (
\frac{e+ef+f}{p}
\right)
\sum_{ y,z \in \mathbb{F}_{p}}
e\left (
\frac{dfz(-f^2(y+1)^{3}+e^2 y^{3}+1)}{p}
\right) \notag \\
&=\frac{p^2(p-1)}{G_{p}^{2}}
\left (
1+\left (
\frac{-1}{p}
\right)
\right)-\frac{p}{G_{p}^{2}}S^{*} \notag\\
&=-\frac{p^{2}(p-1)}{G_{p}^{2}}
\left (
2p\left (
\frac{-1}{p}
\right)
+(p-2)
\left (
\frac{2}{p}
\right)
\right) \notag\\
&=-p(p-1)
\left (
2p+(p-2)
\left (
\frac{-2}{p}
\right)
\right),
\end{align*}
which was what needed to be shown, by \eqref{t3eqa}.
The second equality above follows from Lemma \ref{l1}.
Above $S^{*}$ is as
defined in Lemma \ref{l1in} and in the next-to-last equality we used
Corollary \ref{c1}, part (ii). In the last equality we used once again the fact that
$p/G_{p}^{2}(-1|p) =1$.
\begin{flushright}
$\Box$
\end{flushright}

\section{Concluding Remarks}
Let $p \equiv 5$ $($mod 6$)$ be prime, $b \in \mathbb{F}_{p}^{*}$
and $k$ be an odd positive integer. Define
\[
f_{k}(p)=\sum_{t=0}^{p-1}a_{p,\,t,\,b}^{k}\left(\frac{b}{p}\right).
\]
(It is not difficult to show that the right side is independent of $b \in \mathbb{F}_{p}^{*}$)

By Theorem \ref{t3}
\[
f_{3}(p) =-p\left((p-2)\left(\frac{-2}{p}\right)
+2p\right).
\]
We have not been able to determine $f_{k}(p)$ for $k \geq 5$ (We do not consider
even $k$, since a formula for each even $k$ can be derived from Birch's work in
\cite{B67}). We conclude with a table of values of $f_{k}(p)$
and small primes $p \equiv 5$ $($mod 6$)$, with the hope of encouraging others to
work on this problem.

\begin{table}[!ht]
  \begin{center}
    \begin{tabular}{| c | c | c | c | c |}
    \hline
 $p \diagdown k$     &  5   &       7   &       9   &  11 \\ \hline
     &       &      &       &  \\
5     &     -275 &   -2315    & -20195      &-179195 \\
     &       &      &       &  \\
11   &  -10901     &-358061              &   -12030821    & -411625181 \\
     &        &     &       &  \\
17   &   -36737   & -1582913         &-68613377       &-3016710593  \\
     &       &      &       &  \\
23   &  8257  &  2763745   &304822657      &27903893665 \\
     &       &      &       &  \\
29  &   -35699   &-396299     & 184745341      &35260018501 \\
     &       &      &       &  \\
41   &  -654401   &-88683041             &-12260782721       &-1716248660321  \\
\hline
    \end{tabular}
\phantom{asdf}\\

\vspace{10pt}

    \caption{$f_{k}(p)$ for small primes $p \equiv 5$ $($mod 6$)$
and small odd $k$. }\label{Ta:t1}
    \end{center}
\end{table}

 \allowdisplaybreaks{

}
\end{document}